\documentclass{amsart}
\usepackage{amssymb,amscd,verbatim,latexsym}
\usepackage{enumerate, caption}

\usepackage[mathscr]{eucal}
\usepackage{tikz}
\usetikzlibrary{decorations.pathreplacing}

 \usepackage[%
   colorlinks, 
  citecolor=darkgreen, linkcolor=darkblue,
   pdfpagelabels=true,
   unicode=true
 ]{hyperref}

\usepackage[normalem]{ulem}

\newcommand\define{\mathrel{:= }}
\newcommand\ede{\define}
\newcommand\seq{=}

\newcommand\esssup{\mathop{\mathrm{ess}\kern0.08em\textrm{-}\kern0.08em\mathrm{sup}}}

\definecolor{darkgreen}{HTML}{009900}
\definecolor{darkblue}{HTML}{002299}
\newcommand\II{\mathrm{I}\hskip-.3mm\mathrm{I}}

\renewcommand\ne[1]{|#1|}

\newcommand\datver[1]{\def\datverp
{\par\boxed{\boxed{\text{Version: #1; Run: \today}}}}}\datver{0.1}

\newcommand{\dist}{\operatorname{dist}}

\newcommand{\CC}{\mathbb C}

\newcommand{\NN}{\mathbb N}

\newcommand{\RR}{\mathbb R}

\newcommand{\ZZ}{\mathbb Z}

\newcommand{\CIc}{{\mathcal C}^{\infty}_{\text{c}}}
\newcommand\pa{{\partial}}

\newcommand{\End}{\operatorname{End}}

\newcommand{\vol}{\operatorname{vol}}
\newcommand{\dvol}{\operatorname{dvol}}

\newcommand{\maD}{\mathcal D}

\newcommand{\maH}{\mathcal H}

\newcommand\<{\langle}
\renewcommand\>{\rangle} 
\newcommand{\ie}{\emph{i.\thinspace e.\ }}

\newcommand{\rinj}{\mathop{r_{\mathrm{inj}}}}

\newcommand\partialDM{\partial_{\kern-.04em D}\kern-.08em M}
\newcommand\partialNM{\partial_{\kern-.04em N}\kern-.08em M}
\newcommand\partialD{\partial_{\kern-.04em D}\kern-.08em}
\newcommand\partialN{\partial_{\kern-.04em N}\kern-.08em}

\newcommand\unifexp{\ignorespaces}
\newtheorem{theorem}{Theorem}[section]
\newtheorem{proposition}[theorem]{Proposition}
\newtheorem{corollary}[theorem]{Corollary}

\newtheorem{lemma}[theorem]{Lemma}

\theoremstyle{definition}
\newtheorem{definition}[theorem]{Definition}
\theoremstyle{remark}
\newtheorem{remark}[theorem]{Remark}
\newtheorem{example}[theorem]{Example}

\author[B. Ammann]{Bernd Ammann} \address{B. Ammann, Fakult\"at f\"ur
  Mathematik, Universit\"at Regensburg, 93040 Regensburg, Germany}
\email{bernd.ammann@mathematik.uni-regensburg.de}

\author[N. Gro{\ss}e]{Nadine Gro{\ss}e} \address{N. Gro{\ss}e,
  Mathematisches Institut, Universit\"at Freiburg, 79104 Freiburg,
  Germany} \email{nadine.grosse@math.uni-freiburg.de}

\author[V. Nistor]{Victor Nistor} \address{V. Nistor, Universit\'{e}
  de Lorraine, CNRS, IECL, F-57000 Metz, France
and Inst. Math. Romanian Acad.  PO BOX 1-764, 014700 Bucharest Romania}
\email{victor.nistor@univ-lorraine.fr}

\thanks{B.A. and N.G. have been partially supported by SPP 2026
  (Geometry at infinity), funded by the DFG.  B.A. has also been
  partially supported by the DFG SFB 1085 (Higher
  Invariants). V.N. has been partially supported by
  ANR-14-CE25-0012-01 (SINGSTAR).\\
Manuscripts available from \textbf{http:{\scriptsize
    //}iecl.univ-lorraine.fr{\scriptsize
    /}$\tilde{}$Victor.Nistor{\scriptsize /}}
}

\date{Nov 19, 2018}

\begin{document}

\title[The strong Legendre condition and well-posedness]{The strong
  Legendre condition and the well-posedness of mixed Robin problems on
  manifolds with bounded geometry}

\begin{abstract}
Let $M$ be a smooth manifold with boundary and bounded geometry,
$\pa_D M \subset \pa M$ be an \emph{open and closed} subset of the
boundary of $M$, $P$ be a second order differential operator on $M$,
and $b$ be a first order differential operator on $\pa M$.  Our
operators act on sections of a vector bundle $E \to M$ with bounded
geometry. We prove the regularity and well-posedness in the Sobolev
spaces $H^s(M; E)$, $s \geq 0$, of the \emph{mixed Dirichlet-Robin
  boundary value problem} $$Pu = f \mbox{ in } M,\ u = 0 \mbox{ on }
\pa_D M,\ \pa^P_\nu u + bu = 0 \mbox{ on } \pa M \setminus \pa_D M$$
under the following four natural assumptions. First, we assume that
$P$ satisfies the strong Legendre condition and the first order terms
are small.  (In the scalar case, the strong Legendre condition reduces
to the uniformly strong ellipticity condition.)  Second, we assume
that all the coefficients of $P$ and all their covariant derivatives
are bounded. Third, we assume that $\Re\, b \ge 0$ and that
  there is $\epsilon > 0$ and an open and closed subset $\pa_R M
  \subset \pa M \setminus \pa_D M$ such that $\Re \, b \geq \epsilon
  I$ on $\pa_R M$. Finally, we assume that the distance to $\pa_R M
\cup \pa_D M$ is uniformly bounded on $M$ and that $\pa_R M \cup \pa_D
M$ intersects all components of $\pa M$ (\ie $(M, \pa_R M \cup \pa_D
M)$ has \emph{finite width}).

We include also some extensions of our main result in different
directions. First, the finite width assumption is required for the
Poincar\'{e} inequality on manifolds with bounded geometry, a result
for which we give a new, more general proof. Second, we consider also
the case when we have a decomposition of the vector bundle $E$
(instead of a decomposition of the boundary).  Third, we also consider
operators with non-smooth coefficients, but, in this case, we need to
limit the range of $s$. Finally, we also consider the case of
uniformly strongly elliptic operators and discuss the equivalence
between the \emph{uniform Agmon condition} and the G\aa rding
inequality. The main novelty of our results is that they are
formulated on a \emph{non-compact manifold.}
\end{abstract}

\maketitle
{\hypersetup{linkcolor=black}\tableofcontents}

\section{Introduction}

This is the third paper in a series of papers devoted to the spectral
and regularity theory of differential operators on a suitable
non-compact manifold with boundary~$M$ using analytic and geometric
methods. In this paper, we extend the well-posedness result of
\cite{AGN1}, the first paper of the series, from the case of the
Laplace operator to that of operators (or systems) with non-smooth
coefficients satisfying the strong Legendre condition. Considering
systems is important in practical applications.  We also take
advantage of the general regularity results in \cite{GN17}, the second
paper in this series, to obtain results on Robin boundary
conditions. The Robin boundary conditions ``interpolate'' between
Dirichlet and Neumann boundary conditions, so are natural to consider.

We have made an extra effort to make this paper readable independently
of the previous two papers by recalling the main definitions and
results from those papers.

\subsection{Geometric and analytic settings}
We make several assumptions on the geometry and on the operators. Let
us begin by describing our geometric setting. We fix in what follows a
smooth $m$-dimensional Riemannian manifold with \emph{boundary and
  bounded geometry} $(M,g)$, see Definition~\ref{def_bdd_geo}. In
particular, $\pa M$ is smooth and a manifold with bounded geometry in
its own right. Also, we fix a vector bundle $E \to M$ with a metric
and a compatible connection $\nabla^E$.  We let $R^E$ denote the
curvature of the connection $\nabla^E$.  We assume that all the
covariant derivatives $\nabla^kR^E$ are bounded. Moreover, we assume
our boundary to be \emph{partitioned}, that is, that we are given a
disjoint union decomposition
\begin{equation}\label{eq.decomposition}
  \partial M = \partial_D M \sqcup \partial_N M \sqcup \partial_R M
\end{equation}
where $\partial_D M$, $\partial_N M$ and $\partial_R M$ are (possibly
empty) open and closed subsets of $\partial M$ and $\sqcup$ is the
\emph{disjoint union.} The indices of the notation reflect that these
will become the parts of the boundary where we will impose Dirichlet,
Neumann, and Robin boundary conditions, respectively. Let $A \subset
\pa M$. Recall from \cite{AGN1} that $(M, A)$ is said to have
\emph{finite width} if, in addition to the bounded geometry assumption
on $(M,g)$, $\text{dist}(p, A)$ is uniformly bounded in $p\in M$ and
$A$ intersects all boundary components of $\pa M$.

Let us now describe our analytic setting, which, in particular, will
describe our operators. All the differential operators considered in
this paper will be assumed to have bounded, measurable (\ie
$L^\infty$) coefficients. The most important ingredients are a bounded
sesquilinear form $a$ on $T^*M\otimes E$ and a first order
differential operator $b$ on $E|_{\pa M}$. They give rise to an operator 
$\tilde{P}_{(a,b)} \colon H^1(M;
E) \to H^1(M; E)^*$ by
\begin{equation}\label{eq.def.Pab}
  \< \tilde{P}_{(a,b)}(u), v \> \define \int_{M} a(\nabla u, \nabla v)
  \dvol_{g} + \int_{\pa_R M} (bu, v)_E \dvol_{\pa g},
\end{equation}
where $\dvol$ denotes the volume form with respect to the underlying
metric and $\<\ , \ \>$ denotes the pairing between a space $V$ and
its conjugate dual $V^*$.  (The spaces $H^1$ are recalled in the main
body of the paper). See Section~\ref{subsub.bilin.forms} for more
details. We note that Gesztesy and Mitrea have considered also
non-local operators $b$, see \cite{gesztesyMitrea2014a} and the
references therein. Let also $Q$ and $Q_1$ be first order
differential operators acting on sections of $E$. They define linear
maps $\tilde{Q}, \tilde{Q}^*_1 \colon H^1(M; E) \to H^1(M;
E)^*$. First, we let
\begin{equation}
  \label{eq.def.H1D}
  H^1_D(M; E) \define \{\, u \in H^1(M; E)\, \vert \ u = 0 \mbox{ on }
  \pa_D M \, \}.
\end{equation}
Our regularity and well-posedness results will be for the second order
differential operator
\begin{equation}\label{eq.def.tildeP}
  \tilde P \define \tilde P_{(a, b)} + \tilde Q + \tilde Q_1^*\colon
  H^1_D(M; E) \to H^1_D(M; E)^*,
\end{equation}
which encodes also the Robin boundary conditions. Ignoring these
boundary conditions via the restriction $H^1_D(M; E)^* \rightarrow
H^1_0(M; E)^*$, we obtain the ``typical'' second order differential
operator (in divergence form)
\begin{equation}\label{eq.def.P}
  P \define P_{(a, b)} + Q + Q_1^* \colon H^1_D(M; E) \to H^{-1}(M; E)
  \simeq H^1_0(M; E)^*.
\end{equation}
This operator is hence independent of $b$, unlike $\tilde P$. 

We will use the operator $\tilde P$ to study mixed Dirichlet-Robin
boundary conditions, as follows. Let $\nu$ be the outward unit vector
field at the boundary and $\pa_\nu^P$ the conormal derivative
associated to $P$. We consider the \emph{mixed Dirichlet-Robin
  boundary value problem}:
\begin{equation}\label{eq.mixed}
  \left\{\begin{array}{rl} Pu & = \ f \quad\mbox{ in } M,\\ u & = \ 0
  \quad\mbox{ on } \pa_D M, \\
  \pa^P_\nu u + bu & = \ 0 \quad \mbox{ on } \pa_N M \sqcup
  \pa_RM.\end{array} \right.
\end{equation}
The relation between this boundary value problem and $\tilde P$ is
that, in a certain sense that will be made precise below using the
maps $j_k$ of Equation~\eqref{eq.def.jk}, we have that $\tilde P(u) =
(Pu, \pa^P_\nu u + bu)$. (See \cite{GN17} for a more detailed
discussion of the difference between $P$ and $\tilde P$ and the role
of boundary conditions and \cite{Daners2000, gesztesyMitrea2009} for
some related results.) We note that the operator $\pa^P_\nu$ of the
last equation of \eqref{eq.mixed} depends only on $a$ and $Q_1$. If $P
= \Delta$, the Laplacian, then $\pa^P_\nu = \pa_\nu$ is the usual
normal derivative.

As in \cite{BNZ3D1,gesztesyMitrea2014a}, we shall typically assume for
our main results that $P$ satisfies the strong Legendre condition,
which is the condition that the bilinear form $a$ defining the
principal part $P_{(a, b)}$ of the operator $P$ be strongly coercive
(Definition \ref{def.SLC}). For scalar operators, the strong Legendre
condition is equivalent to the uniform strong ellipticity condition,
but, for systems, the strong Legendre condition is more
restrictive. 

If $T$ is a (possibly unbounded) operator on a Hilbert
space $\maH$, we shall write $T \ge \epsilon$ if $(T\xi, \xi) \ge
\epsilon (\xi, \xi)$ for all $\xi$ in the domain of $T$ and denote
$\Re\, T \define \frac{1}{2}(T + T^*)$. Recall that $\sqcup$ denotes the
disjoint union. Our main result is the following well-posedness
result.

\begin{theorem} \label{thm.main} 
Let $\ell\in \mathbb{N}$, $\ell\geq 1$, and assume that:
\begin{enumerate}[(i)]
 
\item $(M, \pa_D M \sqcup \pa_R M)$ has finite width;
 
\item $P  \define  P_{(a, b)} + Q + Q_1^*$ has coefficients in $W^{\ell,
  \infty}(M; \End(E))$ and satisfies the strong Legendre condition;
   
\item $\Re\,  b \ge 0$ and there is $\epsilon > 0$ such that
  $\Re\,  b \ge \epsilon$ on $\pa_R M$.

  \item There is $\delta > 0$ depending on $\epsilon$, $a$, $b$,
  and $(M, g)$ such that $\Re\, (Q + Q_1^*) \ge - \delta$.

\end{enumerate}
Then, for all $k \in \NN$, $1\leq k \le \ell$, $\tilde P u \define (P
u , (\pa_\nu^P u + bu)\vert_{\pa M \setminus \pa_D M})$ defines an
isomorphism
\begin{equation*}
 \widetilde P\colon H^{k+1}(M; E) \cap \{\, u \vert_{\pa_D M} = 0 \,
 \} \to H^{k-1}(M; E) \oplus H^{k-1/2}(\pa M \setminus \pa_D M; E) .
\end{equation*}
\end{theorem}

This theorem follows directly from Theorem~\ref{thm.H1} and
~\ref{thm.regularity}. In fact, it does not matter what $b$ is on
$\pa_D M$. In particular, the condition $\Re \, b \ge 0$ is necessary
only on $\pa M \smallsetminus \pa_D M$.

\subsection{Comments}
The proof of our main result, Theorem \ref{thm.main} combines the
Poincar\'e inequality with regularity. More precisely, by replacing
$H^{k-1}(M; E) \oplus H^{k-1/2}(M \setminus \pa_D M; E)$ with
$H^{1}_D(M; E)^*$ as the range for $\widetilde P$, our theorem makes
sense also for $k = 0$. This pattern of proof follows the classical
case \cite{Agranovich2002, AGN1, BrezisSobolev, daugeBook, KMR,
  LionsMagenes1, NP, Taylor1, vishikGrushin}. Using the trace theorem
\cite{GrosseSchneider2013}, we can also consider non-homogeneous
Dirichlet boundary conditions in $H^{k+1/2}(\pa_D M; E)$. What is
essentially different in the non-compact case is how these two steps
(Poincar\'e inequality and regularity) are dealt with. 

For instance, the Poincar\'e inequality follows from the finite width
assumption, using the results from \cite{AGN1}. We moreover know, from
that paper, that the assumption that $(M, \pa_D M \sqcup \pa_R M)$
has finite width is necessary in general. Indeed, if $M$ is a subset
of $\RR^n$ with the induced metric such that $(M, \pa M)$ is not of
finite width, then the theorem is not true anymore. A
  counterexample is provided by a domain that coincides with a cone
in a neighborhood of infinity. The finite width assumption on $(M,
\pa_D M \sqcup \pa_R M)$ is needed in order to obtain the Poincar\'e
inequality, which implies the special case $k = 0$ of our theorem, see
Theorem \ref{thm.H1} (and is essentially equivalent to it).

For regularity, we can use either positivity (or coercivity) or a
uniform version of the Shapiro-Lopatinski conditions. We refer the
reader to \cite{GN17}, where this issue is dealt with in detail.

The reader may have wondered what happens in the strongly elliptic
case (for systems). In that case, the coercivity (\ie the G\aa rding
inequality) is equivalent to the \emph{uniform Agmon } condition for
the boundary conditions, see Subsection~\ref{sssec.coercive}. If the
uniform Agmon condition is satisfied, then one obtains the analog of
Theorem~\ref{thm.main} for~$P$ replaced with~$P + R$, for some real,
large enough $R > 0$.

\subsection{Earlier results and novelty of our results}
Recently, many results on Robin boundary conditions were obtained,
almost all devoted to \emph{bounded domains with non-smooth
  boundaries.} This is the case with the nice papers by Dancer and
Daners \cite{DancerDaners}, Daners \cite{Daners2000}, and Gesztesy and
Mitrea \cite{gesztesyMitrea2009, gesztesyMitrea2014a}, to which we
refer for more references. As seen from our result, our focus is
rather on \emph{unbounded} domains, but we assume a smooth
boundary. This allows us also to obtain certain regularity results for
our problem that do not make sense in the Lipschitz case. In fact, our
main theorem, Theorem \ref{thm.main}, is new even in the case of pure
Dirichlet or pure Neumann boundary conditions.

One of the new issues that one has to deal with in the case of
unbounded domains is the Poincar\'e inequality. The $L^1$--Poincar\'e
inequality for scalar functions and for $\pa_D M = \pa M$ was proved
in \cite{Sakurai}. The form that we need is in \cite{AGN1}. In view of
its importance and for further applications, we extended the
Poincar\'e inequality from~\cite{AGN1} to functions vanishing on
suitable subsets $A \subset \pa M$ by using a new proof based on
uniform coverings.  The extension is that we no longer assume that~$A$
be an open and closed subset of $\pa M$, but we need a slightly
stronger condition than that of $(M, A)$ being of finite width.

Theorem \ref{thm.main} was proved in \cite{AGN1} for $P = P_{(g,0)} =
\Delta_g \geq 0$, where~$g$ is the metric and $P_{(a, b)}$ is as
defined in Equation \eqref{eq.def.Pab} above. If $P = P_{(a, 0)}$ with
$E = \underline{\CC}$ (the one dimensional trivial bundle) and $a$ is
real and smooth, Theorem \ref{thm.main} then follows also from the
results of \cite{AGN1} by replacing the metric $g$ with the equivalent
metric $a$, since in the scalar case the strong Legendre condition is
equivalent to the condition that $P$ be uniformly strongly elliptic.
The general case, namely $P = P_{(a, b)} + Q + Q_1^*$, where~$Q$ and
$Q_1$ are first order differential operators, presents the following
additional difficulties:
\begin{enumerate}[(i)]
 \item If $Q + Q_1^*\neq 0$, we cannot reduce $P$ to a Laplacian, even
   if $a$ is smooth;
 \item $a$ may be not be real;
 \item $a$ may not be smooth;
 \item\label{four} the bundle $E$ may be topologically non-trivial or
   of higher dimension.
\end{enumerate}
The first three extensions are relatively easy to deal with. We deal
with $Q \neq 0$ or $Q_1 \neq 0$ by assuming that the negative part of
$Q + Q^* + Q_1 + Q_1^*$ is small enough (Condition (iv) of Theorem
\ref{thm.main}). The case when $a$ is not real simply requires to use
a complex version of the Lax-Milgram Lemma. In the case $a$ is not
smooth, we simply restrict the regularity of the resulting
solution. These three extensions of the results in \cite{AGN1}
\emph{do not} follow from the results of that paper, but can be
obtained using the methods of that paper and those in \cite{GN17},
once the additional background material in Section
\ref{sec.preliminaries} is taken into account.

The last extension, \eqref{four}, (to $E$ non-trivial, that is, to
systems) causes the most headaches and, so far, to the best of our
knowledge, is not dealt with in a completely satisfactory way
anywhere. To extend our results to systems, we chose to consider the
condition that $P$ satisfies the strong Legendre condition. This
condition is satisfied by the Hodge Laplacian $dd^* + d^*d$, but is
often too restrictive for applications.  The weaker condition (that
$P$ be uniformly strongly elliptic) is satisfied in many applications,
but it seems that for systems is does not provide results as strong as
the ones that one obtains for scalar equations.  Nevertheless, one can
obtain coercivity under some additional assumptions, see
\ref{sssec.coercive}.

 We have also included
Robin boundary conditions. Except for a few results and definitions
that we recall from~\cite{AGN1,GN17}, the first two papers of this
series, our paper can be read independently of those papers, as we
recall in Section~\ref{sec.preliminaries} the most important
definitions and results from those papers.

\subsection{Contents of the article}
The article is organized as follows. Section~\ref{sec.preliminaries}
is devoted to preliminaries, including a discussion of Sobolev spaces,
of differential operators on Riemannian manifolds from a global point
of view, and to some background material on manifolds with bounded
geometry from~\cite{AGN1}. The proof of the Poincar\'e inequality is
in Section~\ref{sec.poincare}. The last section contains the proofs of
our main results, which, in turn, yield Theorem~\ref{thm.main}. We
also discuss there some extensions of our results in
Subsection~\ref{ssec.ae}, including the uniform Agmon condition.

\section{Background, notation, and preliminary results}
\label{sec.preliminaries}

We recall here some basic material, for the benefit of the reader. We
also use this opportunity to fix the notation. For instance, \emph{$M$
  will always denote a smooth $m$-dimensional Riemannian manifold,
  possibly with boundary.} The metric of $M$ will be denoted by $g$
and the associated volume form will be denoted by $\dvol_g$.  The
boundary is denoted by $\pa M$, and is assumed to be smooth, for
simplicity, although some intermediate results may hold in greater
generality. We assume that the boundary is partitioned as in Equation
\ref{eq.decomposition}. See \cite{AGN1} or \cite{GN17} for concepts
and notation not fully explained here.

\subsection{General notations and definitions}\label{ssec.not} 
We begin with the most standard concepts and some notation.

\subsubsection{Vector bundles} 
Let $E \to M$ be a smooth real or complex vector bundle endowed with
metric~$(.,.)_E$ and a connection
\begin{equation*}
  \nabla^E \colon \Gamma(M; E) \to \Gamma(M; E \otimes T^{*}M).
\end{equation*}
We assume that $\nabla^E$ is metric preserving, which means that
\begin{equation*}
   X ( \xi, \eta )_E = ( \nabla_X \xi, \eta )_E + (\xi, \nabla_X
   \eta )_E.
\end{equation*}

We endow the tangent bundle $TM\to M$ with the Levi-Civita connection.

\begin{definition}\label{def_ttly_bdd_curv} 
A vector bundle $E \to M$ with given connection \emph{has totally
  bounded curvature} if its curvature and all its covariant
derivatives are bounded (that is, $\|\nabla^k R^E\|_\infty < \infty$
for all $k$).  If $TM$ has totally bounded curvature, we shall then
say that $M$ has \emph{totally bounded curvature}.
\end{definition}

\subsubsection{Sobolev spaces}\label{sec_Sob}
Let us recall the basic definitions related to Sobolev spaces. See
\cite{AmannFunctSp, AubinNL98, EichornSobolev, HebeyBook,
  kordyukovLp1, ThalmaierWang} for related results.  The $L^p$-norm
$\|u\|_{L^p(M;E)}$ of a measurable section of $u$ of $E$ is then
\begin{align*}
 \begin{gathered}
 \ \ \|u\|_{L^p(M; E)}^p \ede \int_{M}\, \ne{u(x)}_E^p\, \dvol_g(x)\,,
 \ \text{ if } \ 1\leq p < \infty\,, \text{ and } \\
 \ \ \|u\|_{L^\infty(M; E)} \ede \esssup_{x \in M}\, \ne{u(x)}_E\,,
 \end{gathered}
\end{align*}
as usual. Let $\ell \in \ZZ_+ = \NN \cup \{0\}$. We define $L^p(M; E)
\define \{ u\, \vert \ \|u\|_{L^p(M; E)} < \infty \}$ and
\begin{align*}
 W^{\ell, p} (M; E) \define\{\, u \, \vert \ \nabla\sp{j} u \in L^p(M;
 E \otimes T^{*\otimes j}M)\,,\ \forall\,j \leq \ell \, \} .
\end{align*}
We let $W^{\infty, p} \define \bigcap_{\ell} W^{\ell, p}$.  

If $M$ has a smooth boundary $\pa M$ and $\pa_D M \subset \pa M$ is an
open and closed subset of $\pa M$, we define
\begin{align}\label{eq.def.WmD}
 W^{\ell,p}_{D}(M; E) \ede \text{closure}_{W^{\ell,p}(M; E)}
 {C_c^\infty(M \setminus \pa_D M; E)\,},
\end{align}
the closure in $W^{\ell,p}(M; E)$ of the space of smooth sections of
$E \to M$ that have compact support not intersecting~$\pa_D M$. As
usual, we shall use the notation
\begin{align}\label{eq.def.HlD}
 H^\ell(M; E) \ede W^{\ell, 2} (M; E) \quad \mbox{and} \quad
 H^{\ell}_D(M; E) \ede W^{\ell, 2}_D(M; E)
\end{align}
in the Hilbert space case ($p = 2$). If $\pa_D M = \pa M$, we simply
write $W_0^{\ell, p}(M; E) \define W_D^{\ell, p}(M; E)$ and $H_0^\ell
(M; E) \define W_0^{\ell, 2}(M; E)$. For manifolds with bounded
geometry, these spaces can be characterized using the trace theorem,
see \cite{GrosseSchneider2013}.

As in \cite{gesztesyMitrea2009, GN17}, we denote by $V^*$ the \emph{
  complex conjugate} dual of the Banach space~$V$. If $-s\in \NN$, we
define $H^{s}(M; E) \simeq H_0^{-s}(M; E)^*$.  If $M$ has no boundary
and $s\in \RR$, then the spaces $H^s(M; E)$ are defined by
interpolating the spaces $H^\ell(M; E)$, $\ell\in \ZZ$.  See
\cite{BrezisSobolev, JostBook, LionsMagenes1, Taylor1} for the case of
manifolds with boundary.

\subsection{Differential operators}
We recall now differential operators on manifolds from a global point
of view.

\subsubsection{General definitions}\label{sec:gd}
A \emph{differential operator} of order $k$ is an expression of the
form $P \ede \sum_{j=0}^k a_j \nabla^j,$ with $a_j$ a section of
$\End(E) \otimes TM^{\otimes j}$.  A differential operator $P =
\sum_{j=0}^k a_j \nabla^j$ will be said to \emph{have coefficients
  in} $W^{\ell, \infty}$ for $\ell \in \ZZ_+ \cup \{\infty\}$ if $a_j
\in W^{\ell, \infty}(M; \End(E) \otimes TM^{\otimes j})$ for all $0
\leq j \leq k$.  If $\ell = 0$, we shall say that $P$ has
\emph{bounded} coefficients.  If $\ell = \infty$, we shall say that
$P$ has \emph{totally bounded} coefficients. We then obtain a
  bounded operator
\begin{align*}
  P \seq \sum_{j=0}^k a_j \nabla^j \colon W^{\ell+k, p}(M; E) \, \to
  \, W^{\ell, p}(M; E), \quad \ell \geq 0.
\end{align*}

\subsubsection{Bilinear forms and operators in
  divergence form}\label{subsub.bilin.forms} We now consider
differential operators in divergence form, which will allow us to
treat the Robin boundary conditions on same footing as the Dirichlet
boundary conditions. See \cite{GN17} for more details. See also
\cite{Daners2000,gesztesyMitrea2014a}. Assume that, for each $x\in M$,
we have a sesquilinear map $a_x \colon T^*_x M \otimes E_x \times
T^*_x M \otimes E_x \to \CC$. The family $(a_x)$ defines a section $a$
of the bundle $((T^*M\otimes E) \otimes (T^{*}M\otimes \overline
E))'$. We say that the section $a=(a_x)_{x\in M}$ is a \emph{bounded,
  measurable sesquilinear form on $T^*M \otimes E$} if it is an
$L^\infty$-section of $((T^*M\otimes E) \otimes (T^{*}M\otimes
\overline E))'$. Let us also consider a first order differential
operator $b$ on $E|_{\pa M}$. The \emph{Dirichlet form} $B_{(a,
  b)}$ on $H^1_{D}(M; E)$ associated to such a bounded family of
sesquilinear forms $a$ and endomorphism section $b$ is then
\begin{align}\label{eq.def.B}
 B_{(a, b)}(u, v) \define B_{(a, b)}^g(u, v) \define \int_{M} a(\nabla
 u, \nabla v) \dvol_{g} +  \int_{\pa M \smallsetminus \pa_D M }
   (bu, v)_E \dvol_{\pa g}.
\end{align}
Note that $\<\tilde{P}_{(a,b)}(u),v\>=B_{(a, b)}(u, v)$ by \eqref{eq.def.Pab}. 
If $Q$ is a first order differential operator with bounded
coefficients, then it also defines a continuous map $\tilde Q \colon
H^1_{D}(M; E)\rightarrow L^2(M; E) \subset H_D^{1}(M;E)^*$.  The
adjoint $\tilde Q^*$ of $\tilde Q$ will then map $\tilde Q^* \colon
H^1_{D}(M; E)\rightarrow H_D^{1}(M;E)^*$ as well.
The sesquilinear form $B_{(a, b)}$ and the differential operators
$Q$ and $Q_1$ then define the differential
operators $\tilde P_{(a, b)}$, $\tilde P$, and $P$ of
Equations (\ref{eq.def.Pab}-\ref{eq.def.P}).

\begin{definition}\label{def.divergence}
We shall say that $ \tilde P = \tilde P_{(a, b)} + \tilde Q + \tilde
Q_1^* \colon H^1_{D}(M; E) \longrightarrow H_D^{1}(M;E)^* $ and $P
\define P_{(a, b)} + Q + Q_1^*$ are second order \emph{differential
  operators in divergence form} if $a$ is a bounded, measurable
sesquilinear form on $T^*M \otimes E$, $b = b_1 + b_2$ is a first
order differential operator on $E|_{\pa M}$, with $b_1$ with
$W^{1,\infty}$ coefficients and $b_2$ a bounded, measurable
endomorphism of $E|_{\pa M}$, and $Q$ and $Q_1$ are first order
differential operators with bounded coefficients. In particular, $P$
will have bounded coefficients.
\end{definition}

\begin{remark}
We have by definition
\begin{equation}\label{eq.B.tildeP}
  \< \tilde P u, v \> \define B(u, v) \define(\nabla u, \nabla v) +
  (bu, v)_{\pa M \smallsetminus \pa_D M} + (Qu, v ) + (u, Q_1
  v )
\end{equation} 
where, we recall, $\<.,.\>$ denotes the dual pairing and $(.,.)_N$ denotes the
$L^2$-product on the manifold $N$ (in case $N=M$ we omit the index).
\end{remark}

\subsubsection{Boundary value problems}\label{sssec.bvp}
We are interested in differential operators in divergence form $\tilde
P \colon H^1_D(M; E)\to H^1_D(M; E)^*$ because we have the equivalence of
the following two problems
\smallskip

\begin{enumerate}[(i)]
 \item The operator $\tilde P \colon H^1_D(M; E) \to H^1_D(M; E)^*$ is
   an isomorphism. \smallskip
   
 \item For each $F \in H^1_D(M; E)^*$ and $h \in H^{1/2}(\pa_D M; E)$,
   the ``weak'' problem
 \begin{align*}%\label{eq.form.three}
  \begin{cases}  
  \tilde P(u)(v) = F(v) & \text{for all } v\in H^1_D(M; E)
  \\ \qquad\quad u = h & \text{on } \pa_D M
  \end{cases} 
 \end{align*}
 has a unique solution $u \in H^1(M; E)$, depending continuously on
 $F$ and $h$. \smallskip
 
\hspace*{-15mm} The well-posedness of these problems implies then the
well-posedness of
 \smallskip
   
 \item Let $f \in L^2(M; E)$, $h \in H^{3/2}(\pa_D M; E)$, and $h_1
   \in H^{1/2}(\pa M \setminus \pa_D M; E)$. The boundary value
   problem
 \begin{align}\label{eq.form.two}
 \begin{cases}
    \quad\quad\quad Pu &= f \ \quad \text{in }\ \ M \\
   \ \ \quad\quad\quad u &= h \ \quad \text{on }\ \pa_D M \\
   \ \pa_\nu^P u + bu & = h_1 \quad \text{on }\ \pa M \setminus \pa_D
   M
   \end{cases}
 \end{align}
 has a unique solution $u \in H^2(M; E)$, depending continuously on
 $f$, $h$, and $h_1$. This is obtained by taking $F(v) \define \int_M (f,
 v)_E \dvol_g + \int_{\pa M \setminus \pa_D M} (g, v)_E \dvol_{\pa g}$ and using
 higher regularity. See Corollary \ref{cor.well.posedness.bvp_2}
 below.
\end{enumerate}

\noindent For higher regularity of the data, we obtain the usual
formulation of mixed boundary value problems. See Subsection
\ref{ssec.four}.  In particular, see
\cite{Daners2000,gesztesyMitrea2014a,GN17} for the need of $Q_1^*$ in
the statement of the main theorem (Theorem \ref{thm.main}) and for how
$Q_1^*$ affects the boundary conditions (\ie  the boundary operator
$\pa_\nu^P$). Note that the well-posedness of Problem
\eqref{eq.form.two} implies right away that of Problem
\eqref{eq.mixed}. The converse is also true in view of the trace
theorem of \cite{GrosseSchneider2013}, since $\pa M$ was assumed to be
smooth.

The best way to understand the operator $\pa_\nu^P$ is using 
boundary triples \cite{Behrndt2015,Post2007}. See \cite{Daners2000, GN17} 
for explicit definitions of
$\pa^P_\nu$ in local coordinates.

\subsection{Manifolds with boundary and bounded geometry\label{ssec.bg}}
We first recall some basic material on manifolds with boundary and
bounded geometry from \cite{AGN1}, to which we refer for more details
(see also \cite{ElderingThesis, Schick2001}).  As in \cite{AGN1}, we
will only assume that our manifolds are paracompact (thus we will
\emph{not require} our manifolds to be second countable).

If $x, y \in M$, then $\dist(x, y)$ denotes the distance between $x$
and $y$ with respect to the metric $g$. If $N \subset M$, then
\begin{equation}\label{eq.def.Ur}
   U_r(N) \ede \{x \in M \mid\, \exists y \in N,\, \dist(x, y) < r\,\}
\end{equation}
will denote the $r$-neighborhood of $N$, that is, the set of points of
$M$ at distance $< r$ to $N$. Thus, if $E$ is a Euclidean space, then
$B_r^E(0)\define U_r(\{0\}) \subset E$ is simply the ball of radius
$r$ centered at $0$.

Let $N$ be a hypersurface in $M$, \ie a submanifold with $\dim
N=\dim M -1$.  We assume that $N$ carries a globally defined normal
vector field $\nu$ of unit length, simply called a \emph{unit normal
  field}, which will be fixed from now on.  The Levi--Civita
connection for the induced metric on $N$ will be denoted by
$\nabla^{N}$.  The symbol $\II^N$ will denote the \emph{second
  fundamental form} of $N$ (in $M$: $\II^N(X,Y)\nu \ede \nabla_XY -
\nabla^{N}_XY$).

Let $\exp_p^M\colon T_pM \to M$ be the exponential map at $p$
associated to the metric and
\begin{align*}
\begin{gathered}
  \rinj(p) \ede \sup\{ r \mid \exp_p^M \colon B_r^{T_pM}(0) \to M
  \text{ is a diffeomorphism onto its image} \}\\
 \rinj(M) \ede \inf_{p \in M} \rinj(p).
\end{gathered}
\end{align*}

\begin{definition}%\label{def_bdd_geo_no_bdy}
A Riemannian manifold without boundary $(M,g)$ is said to be of
\emph{bounded geometry} if $\rinj(M)>0$ and if $M$ has \emph{totally
  bounded curvature}.
\end{definition}

If $M$ has boundary, then $\rinj(M)=0$. Let $\exp^\perp(x,t) \ede
\exp^M_x(t\nu_x)$.

\begin{definition}\label{hyp_bdd_geo} 
Let $(M^{m},g)$ be a Riemannian manifold of bounded geometry with a
hypersurface $H = H^{m-1}\subset M$ and a unit normal field $\nu$ on
$H$. We say that $H$ is a \emph{bounded geometry hypersurface} in $M$
if the following conditions are fulfilled:
\begin{enumerate}[(i)]
\item $H$ is a closed subset of $M$;
\item $\|(\nabla^H)^k \II^H \|_{L^\infty} < \infty$ for all $k \geq 0$;
\item $\exp^\perp\colon H\times (-\delta,\delta)\to M$ is a
  diffeomorphism onto its image for some $\delta > 0$.
\end{enumerate}
\end{definition}

As we have shown in \cite{AGN1}, we have that the Riemannian manifold
$(H, g|_H)$ in the above definition is then a manifold of bounded
geometry. See also \cite{ElderingCR, ElderingThesis} for a larger
class of submanifolds of manifolds with bounded geometry. We shall
denote by $r_\pa$ the largest value of $\delta$ satisfying the last
condition of the last definition. Recall from \cite{Schick2001} the
following definition (the precise form below is from \cite{AGN1}):

\begin{definition}\label{def_bdd_geo}  
 A Riemannian manifold~$M$ with (smooth) boundary has \emph{bounded
   geometry} if there is a Riemannian manifold $\widehat M$ with
 bounded geometry satisfying
\begin{enumerate}[(i)]
\item $M$ is contained in $\widehat M$;
\item $\partial M$ is a bounded geometry hypersurface in $\widehat M$.
\end{enumerate}
\end{definition}

\begin{example} Lie manifolds have bounded geometry 
\cite{sobolev, aln1}. It follows that Lie manifolds with boundary are
manifolds with boundary and bounded geometry.
\end{example}

For our well-posedness results, we shall also need to assume
  that $M \subset U_R(\partial_D M \cup \partial_R M )$, for some $0 <
  R < \infty$, and hence, in particular, that $\partial_D M \cup
  \partial_R M \neq \emptyset$.

\begin{definition}\label{def.finite.width}
If $M$ is a manifold with boundary and bounded geometry, if $A \subset
\pa M$ and $M \subset U_R(A)$, for some $0 < R < \infty$, we shall say
that $(M, A)$ \emph{has finite width}.
\end{definition}

Since we let $\dist(x, y) = \infty$ if $x$ and $y$ belong to different
components of $M$, the condition that $(M, A)$ have finite width then
implies, in particular, that $A$ intersects every component of
$M$. See \cite{AmannMaxReg, BaerGinoux, BaerWafo, ElderingThesis,
  gerardBG, GerardWrochna, Krainer} for applications of manifolds of
bounded geometry.

Vector bundles with totally bounded curvature defined on manifolds
  with bounded geometry are called \emph{vector bundles with bounded
    geometry.}

\section{The Poincar\'e inequality}\label{sec.poincare}

We now give a new proof of the Poincar\'e inequality in \cite{AGN1,
  Sakurai} and generalize it by allowing more general subsets of the
boundary where the function vanishes. We assume from now on that
$M$ is a Riemannian manifold with boundary and bounded geometry.

\subsection{A uniform Poincar\'e inequality for bounded domains}
We shall need the following extension of the Poincar\'e inequality
(see \cite{BrennerScott, Ciarlet} or \cite[\S 5.8.1]{EvansBook}),
which is proved (essentially) in the same way as the classical result.

\begin{proposition}\label{prop.gen.poincare}
Assume that $\Omega$ is a connected domain of finite volume in a
Riemannian manifold $(M,g)$ such that $H^1(\Omega) \to L^2(\Omega)$ is
a compact operator. Let $K \subset L^2(\Omega)^*$ be a bounded, weakly
closed set of continuous linear functionals such that $L(1) \neq 0$
for all $L \in K$. Then there is $C > 0$ such that, for any $f \in
H^1(\Omega)$ and any $L \in K$, we have
\begin{equation*}%\label{eq.PoincareG}
  \int_{\Omega} |f|^2 \dvol_g\, \leq \, C \Big ( \, \int_{\Omega}
  |\nabla f|^2 d \vol_g + |L(f)|^2 \, \Big ) .
\end{equation*}
\end{proposition}

\begin{proof} 
Let us assume, by contradiction, that the contrary is true. Then we
can find a sequence $f_n \in H^1(\Omega)$ and a sequence $L_n \in K$
such that
\begin{equation}\label{eq.wrong}
  \int_{\Omega} |f_n|^2 d \vol_g\, > \, n \Big ( \, \int_{\Omega}
  |\nabla f_n|^2 d \vol_g + |L_n (f_n)|^2 \, \Big ) .
\end{equation} 
By replacing $f_n$ with $\|f_n\|_{H^1(\Omega)}^{-1}f_n$, we may assume
that $\|f_n\|_{H^1(\Omega)} = 1$. Then Equation \eqref{eq.wrong} gives
that $\nabla f_n \to 0$ in $L^2(\Omega)$ in norm and that $L_n(f_n)
\to 0$.

Since the unit ball in a Hilbert space is a weakly compact set (by the
Alaoglu-Bourbaki theorem) and we are dealing with a separable Hilbert
space (so the weak topology on the unit ball of $H^1(\Omega)$ is
metrisable), the sequence $f_n$ has a subsequence weakly converging in
$H^1(\Omega)$ to some $v \in H^1(\Omega)$. We replace the original
sequence with that sequence.  Then $\nabla f_n$ converges weakly to
$\nabla v$ in $L^2(\Omega)$, since $\nabla \colon H^1(\Omega) \to
L^2(\Omega)$ is continuous. Therefore $\nabla v = 0$ since we have
seen that $\nabla f_n \to 0$ in $L^2(\Omega)$ in norm. Since $\Omega$
is connected, it follows that $v$ is a constant.

Since $H^1(\Omega) \to L^2(\Omega)$ is a compact operator (by
assumption), we obtain that $f_n$ converges to $v$ in norm in
$L^2(\Omega)$. Since $K$ was assumed to be bounded and weakly closed,
it is weakly compact. We thus have that, by passing to a subsequence,
we may also assume that $L_n$ converges weakly to some $L \in K$. We
thus obtain that $L_n(f_n) \to L(v)$, and hence $L(v) = 0$. Since $v$
is a constant and $L(1) \neq 0$ (since $L \in K$), we obtain $v =
0$. This gives $$1 = \|f_n\|_{H^1(\Omega)}^2 = \|f_n\|_{L^2(\Omega)}^2
+ \|\nabla f_n\|_{L^2(\Omega)}^2 \to \|v\|_{L^2(\Omega)}^2 + 0 = 0,$$
which is a contradiction.
\end{proof}

% ADDITION >> 
We can replace the assumption that $K \subset L^2(\Omega)^*$ be a
bounded, weakly closed set of continuous linear functionals with the
assumption that $K \subset H^1(\Omega)^*$ be a (norm) compact subset.
We shall need the following two corollaries (which hold in greater
generality, but, for simplicity, we state the case that we need).

\begin{corollary}\label{cor.uniform}
 Let $\Omega$ be an open ball in $\RR^n$ and $\epsilon > 0$. Then
 there exists $C > 0$ such that, for any $B \subset \Omega$ a subset
 of measure $\ge \epsilon$, we have
 \begin{equation*}
   \int_{\Omega} |f|^2 dx\, \leq \, C \Big ( \, \int_{\Omega} |\nabla
   f|^2 dx + \int_{B} |f|^2 dx \, \Big )
 \end{equation*}
for all $f\in H^1(\Omega)$.
\end{corollary}

\begin{proof}
We consider $K \define \{L \in L^2(\Omega)^* \, \vert\ \|L\| \le
\vol(\Omega)^{\frac{1}{2}}, L(1) \geq \epsilon \}$ which is norm
closed, convex, and bounded. Hence it is weakly compact.  Then $L(u)
\define \int_B u \dvol_g$ is in $K$, whenever $B \subset \Omega$ is a
subset of measure $\ge \epsilon$. Proposition \ref{prop.gen.poincare}
then gives
\begin{equation*}%\label{eq.PoincareG}
   \int_{\Omega} |f|^2 dx\, \leq \, C \Big ( \, \int_{\Omega} |\nabla
   f|^2 dx + \big|\int_{B} fdx \big|^2 \, \Big ) ,
\end{equation*}
for some $C$ independent of $f \in H^1(\Omega)$ and of $B$ (of measure
$\ge \epsilon$).  The result then follows from the Cauchy-Schwarz
inequality applied to $f$ and the characteristic function of $B$: $|
\int_{B} fdx |^2 \le \Big ( \int_{B} dx \Big ) \, \int_{B} |f|^2 dx
\le \vol(\Omega) \int_{B} |f|^2 dx$.
\end{proof}

Similarly, we obtain the following corollary.

\begin{corollary}\label{cor.uniform2}
Let $\Omega$ be an open ball in $\RR^n$, $\Omega = \Omega' \times
[0, r]$, and $\epsilon > 0$. Then there exists $C > 0$ such that, for
any $B \subset \Omega \times \{0\}$ a subset of measure $\ge
\epsilon$, we have
\begin{equation*}
   \int_{\Omega} |f|^2 dx\, \leq \, C \Big ( \, \int_{\Omega} |\nabla
   f|^2 dx + \int_{B} |f|^2 dx' \, \Big ) .
\end{equation*}
\end{corollary}

\subsection{Proof of the Poincar\'e inequality}
Next we globalize the above inequalities to manifolds $M$ with
boundary and bounded geometry.  We assume --- following
Definition~\ref{def_bdd_geo} --- that $M$ is embedded in a
boundaryless manifold $\widehat M$ of the same dimension, of bounded
geometry and without boundary, such that $\partial M$ is a bounded
geometry hypersurface in $\widehat M$. Recall that $U_r(A)$ denotes
the set of points of $M$ at distance $<r$ to $A$. We use the notation
in \cite{AGN1}, which we recall now: Let $\{p_\gamma\}_{\gamma \in I}$
be a subset of $M$ and $0 < 3r < \min\{\rinj(M), \rinj(\pa M),
r_{\pa}\}$, where~$r_\pa$ is the largest value of $\delta$ satisfying
the last condition of Definition \ref{hyp_bdd_geo} for $H = \pa M$ and
for $M$ replaced with $\widehat M$.  We let $W_\gamma \define
W_\gamma(r) \define U_r(p_\gamma)$, if $p_\gamma$ is an interior point
of $M$; otherwise we let $W_{\gamma} \define W_{\gamma}(r) \define
\exp^{\perp}(B^{T\pa M}_{r}(0) \times [0,r))$.

\begin{definition}\label{def.FC}  
Let $(M^m,g)$ be a manifold with boundary and bounded geometry and
assume $0 < 3r < \min\{\rinj(M), \rinj(\pa M), r_{\pa}\}$ as above.  A
subset $\{p_\gamma\}_{\gamma \in I}$ is called an \emph{\unifexp
  $r$-covering subset of $M$} if the following conditions are
satisfied:
\begin{enumerate}[(i)]
\item For each $R>0$, there exists $N_R \in \NN$ such that, for each
  $p \in M$, the set $\{\gamma \in I\vert\, \dist(p_\gamma, p) < R\}$
  has at most $N_R$ elements.
 \item For each $\gamma \in I$, we have either $p_\gamma \in \pa M$ or
   $d(p_\gamma, \pa M) \geq r$.
 \item $M \subset \bigcup_{\gamma = 1}^{\infty} W_{\gamma}(r)$.
\end{enumerate}
\end{definition}

We have the following Poincar\'e-type inequality, which allows us to
consider more general Dirichlet boundary conditions.

\begin{theorem}\label{thm.Poinc1}
Let $(M,g)$ be a Riemannian manifold with boundary of bounded
geometry, $E \to M$ be a hermitian vector bundle with a metric
preserving connection, and $A \subset \pa M$ be a measurable
subset. We assume that there exists an \unifexp $r$-covering subset
$\{p_\gamma\}_{\gamma \in I}$ and $S \subset \{\gamma \in I\,
|\ p_\gamma\in \pa M\}$ satisfying the following properties:
\begin{enumerate}[(i)]
\item $\dist(x, S)$ is bounded on $M$;
 
\item there exists $\epsilon > 0$ such that, for any $\gamma \in S$,
  $\vol_{\pa M}(A \cap W_\gamma) \geq \epsilon \vol_{\pa
  M}(W_\gamma)$.
\end{enumerate} 
Then there exists $C_{M, A} > 0$ such that
\begin{equation*}
 \int_{M} |f|^2 \dvol_g \leq C_{M, A} \Big( \int_{M} |\nabla f|^2
 \dvol_g + \int_{A} |f|^2 \dvol_{\pa g} \Big),
\end{equation*}
for any smooth, compactly supported section $f$ of $E$.
\end{theorem}

\begin{proof} 
The vector bundle case follows from the scalar case by Kato's
inequality, see the end of the proof. So let us assume in the
beginning that $E$ is the trivial real line bundle and hence that $f$
is a smooth, real-valued, compactly supported function on~$M$.

Let us assume, for the simplicity of notation, that we have a
countable set of indices $\gamma$ for our \unifexp $r$-covering set,
which is equivalent to having a countable basis of topology.  We first
enlarge the given set $\{p_\gamma\}$ to be an \unifexp $r/3$-covering
set (but still use $r$ to define the sets $W_\gamma$; we need that in
order to ensure that two neighboring $W_\gamma$'s will meet in a large
enough set). Let $S_0\define S \subset \NN$. Define, by induction,
$S_{\ell+1}$ to be the set of $\gamma \in \NN \setminus
\bigcup_{j=0}^{\ell} S_{j}$ such that $p_\gamma$ is at distance at
most $2r/3$ to $S_{\ell}$. Then, for $N$ large enough, we have $\NN =
S_0 \cup S_1 \cup \ldots \cup S_N$, since there exists (by assumption)
$R>0$ such that $\dist(x, S) \leq R$, for all $x \in M$.  For each
$\gamma \in S_{\ell +1}$, $\ell \geq 0$, we choose a
\emph{predecessor} $\pi(\gamma) \in S_{\ell}$ such that
$\dist(p_\gamma, p_{\pi(\gamma)}) \leq 2r/3$.

Below, $C >0$ is a constant (close to 1) that yields a comparison of
the volume elements on $M$ and on the coordinate charts
$\kappa_\gamma \define  \kappa_{p_\gamma}$ corresponding to the \unifexp
$r$-covering defined by the \unifexp $r$-covering set
$\{p_\gamma\}_{\gamma \in \NN}$:
\begin{align}\label{eq.FC-chart}
  \left\{\begin{matrix} \kappa_p \colon B^{m-1}_{r}(0) \times [0,r)\to
    M,\hfill & \kappa_p(x, t) \define \exp^{\perp}(\exp_{p}^{\pa
      M}(x), t),\hfill & \text{ if } p \in \pa M\\
   \kappa_p \colon B^{m}_{r}(0) \to M,\hfill & \kappa_p(v) \define
   \exp_p^{M}(v),\hfill & \text{ otherwise.}
% (if } \dist(x, \pa M) \geq r).
\end{matrix}\right.
\end{align}
(So $W_\gamma$ is the image of $\kappa_{p_\gamma}$.) The constant $C$
depends only on $r$ and $M$, but not on $\gamma \in \NN$, since $M$
has bounded geometry and we have chosen $r$ less than the injectivity
radius of $M$. If $\gamma \in S_0\define S$, then
Corollary~\ref{cor.uniform2} gives for $\Omega\define B_{r}^{m-1}(0)
\times [0, r)$ and $B \define \kappa_\gamma^{-1}(A \cap W_\gamma)
  \subset B_{r}^{m-1}(0) \times \{0\}$
\begin{multline}\label{eq.one}
  \int_{W_\gamma} |f|^2 \dvol_g \leq C \int_{\Omega} |f \circ
  \kappa_\gamma|^2 dx \leq CC_{\Omega} \Big ( \int_{\Omega}
  |\nabla^{E} (f \circ \kappa_\gamma)|^2 dx + \int_{B} |f \circ
  \kappa_\gamma|^2 \dvol_{\pa g} \Big ) \\
  \leq C C' C_{\Omega} \Big ( \int_{W_\gamma} |\nabla f|^2 dx +
  \int_{W_\gamma \cap A} |f|^2 \dvol_{\pa g} \Big ).
\end{multline}
Here $\nabla^E$ is the covariant derivative with respect to the
euclidean metric and $C'$ is the constant in the equivalence of the
local $H^1$-norms with respect to the euclidean metric and $g$. Again,
since $(M,g)$ is of bounded geometry, this constant does not depend on
$\gamma$.  On the other hand, by the bounded geometry assumption and
the choice of the $W_\gamma$, if $\gamma \notin S_0$, the sets
$W_\gamma$ and $W_{\beta}$ will intersect in a set of volume (or
measure) $\ge \epsilon$ for some $\epsilon > 0$ independent of
$\gamma$ and $\beta$ if $\dist(p_\gamma, p_\beta) \leq 2r/3$.  Then
using Corollary~\ref{cor.uniform} (for $\Omega\define B_{r}^{m}(0)$
and $B\define \kappa_\gamma^{-1}(W_\gamma\cap W_\beta) \subset
B_{r}^{m}(0)$, and $\beta = \pi(\gamma)$) and similar calculations to
\eqref{eq.one} we obtain
\begin{equation}\label{eq.two}
  \int_{W_\gamma} |f|^2 \dvol_g \leq C \Big ( C_{\Omega}
  \int_{W_\gamma} |\nabla f|^2 \dvol_g + \int_{W_{\pi(\gamma)}} |f|^2
  \dvol_g \Big ).
\end{equation}

Iterating Equation \eqref{eq.two} and using Equation \eqref{eq.one} we
obtain that there exists $C_k > 0$ such that for $\gamma\in S_k$ we
have
\begin{equation}\label{eq.three}
  \int_{W_\gamma} |f|^2 \dvol_g \leq C_k \Big ( \sum_{j=0}^k
  \int_{W_{\pi^j(\gamma)}} |\nabla f|^2 \dvol_g +
  \int_{W_{\pi^k(\gamma)} \cap A} |f|^2 \dvol_{\pa g} \Big).
\end{equation}
(This equation reduces to Equation \eqref{eq.one} if $k = 0$.)  Since
our cover is uniform, there exists $N_0 \in \NN$ such that no point in
$M$ belongs to more than $N_0$ sets of the form~$W_{\gamma}$.  We can
also assume that the $C_0 \leq C_1 \leq \ldots \leq C_N$ (recall that
we stop at $N$). Using these observations and summing up
\eqref{eq.three} over $\gamma$, we obtain
\begin{align*}
  \int_{M} |f|^2 \dvol_g &\leq \sum_{\gamma = 1}^\infty
  \int_{W_\gamma} |f|^2 \dvol_g \\
  & \leq C_N \sum_{\gamma = 1}^\infty \Big ( \sum_{j=0}^k
  \int_{W_{\pi^j(\gamma)}} |\nabla f|^2 \dvol_g +
  \int_{W_{\pi^k(\gamma)} \cap A} |f|^2 \dvol_{\pa g} \Big ) \\
  & \leq (N+1) N_0 C_N \Big ( \int_{M} |\nabla f|^2 \dvol_g + \int_{A}
  |f|^2 \dvol_{\pa g} \Big),
\end{align*}
which is the desired inequality in the scalar case where $C_{M, A} =
(N+1)N_0 C_N$ (note that $k$ depends on $x$, but we have $k \le N$,
which explains the factor $N+1$ in $C_{M,A}$).

For general vector bundles $E$ with metric connection, we have the
Kato inequality $ | \nabla |f|_E| \leq | \nabla f|_E$. Using then the
inequality just proved for $f$ replaced by $|f|$ we have
\begin{align*}
  \int_{M} |f|^2 \dvol_g &\leq C_{M, A} \, \Big( \int_{M} |\nabla |f|
  |^2 \dvol_g + \int_{A} |f|^2 \dvol_{\pa g}\, \Big ) \\
  & \leq C_{M, A} \Big( \int_{M} |\nabla f|^2 \dvol_g + \int_{A} |f|^2
   \dvol_{\pa g} \Big).
\end{align*}
This completes the proof.
\end{proof}

\begin{example}
Let $M = [0, 1] \times \RR$. Then $A = \bigcup_{k \in \ZZ}\, \{0 \}
\times [2k, 2k+1]$ satisfies the assumptions of our theorem, however,
that would not be the case if we replaced $A$ with $\{0\} \times [0,
  \infty)$.
\end{example}

We obtain the following result (proved for $A = \pa M$ in \cite{Sakurai} and,
in general, for $A=\pa_D M$ in \cite{AGN1})

\begin{corollary}\label{cor.Poinc1}
Let $(M,g)$ be a Riemannian manifold with boundary of bounded
geometry, let $A \subset \pa M$ be an open and closed subset such that
$(M, A)$ has finite width. Moreover, let $E \to M$ be a hermitian
vector bundle with a metric preserving connection. Then there exists
$C_{M, A} > 0$ such that
\begin{equation*}
  \int_{M} |f|^2 \dvol_g \leq C_{M, A} \Big( \int_{M} |\nabla f|^2
  \dvol_g + \int_{A} |f|^2 \dvol_{\pa g} \Big),
\end{equation*}
for any smooth, compactly supported section $f$ of $E$.
\end{corollary}

\begin{proof}
This follows right away from Theorem \ref{thm.Poinc1} by taking any
\unifexp $r$-covering set $\{p_\gamma\}$ and $S = \{\gamma\,
\vert\ p_\gamma \in A\}$.
\end{proof}

We have the following extension of the Poincar\'e inequality

\begin{corollary}
Let us keep the assumptions of Corollary \ref{cor.Poinc1}.  Then
\begin{equation*}
 \int_{M} |f|^2 \dvol_g \leq C_{M, A}^k \, \int_{M} |\nabla^k f|^2
 \dvol_g ,
\end{equation*}
for any $f \in H^k(M; E)$, vanishing of order $k$ at $A$.
\end{corollary}

\begin{proof}
Both the left hand side and the right hand side are continuous with
respect to the $H^k$-norm. We have that $\CIc(M \setminus A; E)$ is
dense in $\{ f \in H^k(M; E)\, | \ \pa_\nu^j u = 0 \mbox{ on } A, j
\le k-1\} $ (see \cite{AGN1} and the references therein, for
instance). The proof is then obtained by iterating
Corollary~\ref{cor.Poinc1}.
\end{proof}

%%%%%%%%%%%%%%%%%%%%%%%%%%%%%%%%%%%%%%%%%%%%%%%%%%%%%%%%%%%%%%%%%%%%%%%%%%%%%%%%%%%%%%%%%%

\section{Well-posedness}
\label{sec.five}

We now prove our well-posedness results, under the assumption that $P$
satisfies the strong Legendre condition, that $(M, \pa_D M \cup \pa_R
M)$ has finite width, and that $E \to M$ has totally bounded curvature
(in which case, we recall, $E$ is said to have bounded geometry). See
Subsection \ref{ssec.ae} for an extension of our results to the case
when we have a decomposition $E \vert_{\pa M} = E_D \oplus E_R \oplus
E_N$ of the vector bundle $E\vert_{\pa M}$, instead of a decomposition
of the boundary $\pa M$.

Recall that, by the definition of finite width, our assumption that
$(M, \pa_D M \cup \pa_R M)$ has finite width implies, in particular,
that $M$ is of bounded geometry. Also, recall that we assume that all
our differential operators have bounded coefficients.

\subsection{Coercivity} In order to study the invertibility of operators 
like $\tilde P$, one often uses ``strong coercivity.'' An easy way to
obtain strongly coercive operators is to combine the ``strong Legendre
condition'' with the Poincar\'e inequality. See, however, Subsection
\ref{ssec.ae} for a discussion of uniformly strongly elliptic
operators and of the G\aa rding's inequality. We now recall the needed
concepts, using the terminology of~\cite{AgmonCoercive, ChenWu}.  See
also \cite{gesztesyMitrea2009, gesztesyMitrea2014a,
  McLeanBook,Taylor1}.

\begin{definition} \label{def.SLC}
Let $a$ be a bounded, measurable sesquilinear form on $T^*M \otimes
E$. We say that $a$ satisfies the \emph{strong Legendre condition} if
there exists $\gamma_a>0$ such that
\begin{align}\label{eq.u.s.e}
  \Re\, a(\zeta , \zeta ) \, \geq\, \gamma_a |\zeta|^{2} , \text{ for
    all } \zeta \in T^{*}M \otimes E\,.
\end{align}
\end{definition}

Note that this is a condition at every $T_x^*M \otimes E_x$ and that
it is \emph{uniform} in $x$. It would be more apropriate then to say
that $a$ satisfies the \emph{uniform strong Legendre condition}. For
simplicity, we have chosen not to do that. However, in agreement with
the standard terminology, we use the terminology \emph{uniformly strongly
elliptic} for operators that are strongly elliptic with uniform
constants. We can now introduce the operators in which we are
interested.

\begin{definition} \label{def.u.s.e}
Let $\tilde P = \tilde P_{(a, b)} + \tilde Q + \tilde Q_1^*$ be a
second order (linear) differential operator in divergence form on the
vector bundle $E \to M$ (Definition \ref{def.divergence}), with $Q$
and $Q_1$ first order differential operators (as usual). We shall say
that $\tilde P$ (or $P$) satisfies the \emph{strong Legendre
  condition} if $a$ does. (Recall that it is a standing assumption
that $\tilde P$ has bounded coefficients.)
\end{definition}

Thus $P$ satisfies the strong Legendre condition if, and only if,
$P_{(a, b)}$ does. Moreover, if $P$ satisfies the strong Legendre
condition, then it is uniformly strongly elliptic.  One of our results
next amounts to the fact that, if the Poincar\'e inequality is
satisfied, if $P = P_{(a, b)}$ satisfies the strong Legendre
condition, if $\Re\, b \geq \epsilon$, $\epsilon > 0$ on $\pa_R
  M$ and $\Re\, b \geq 0$ on $\pa M$, and if condition (iii) of
Theorem~\ref{thm.main} is fulfilled, then $P$ will also be ``strongly
coercive,'' a concept that we now recall.

\begin{definition}\label{def.s.coercive}
Let $V$ be a Hilbert space and let $S \colon V \to V^{*}$ be a bounded
operator. We say that $S$ is \emph{strongly coercive} (on $V$) if
there exists $\gamma > 0$ such that
\begin{equation*}
 \Re\, \<Su, u\> \geq\gamma \|u\|_V^2.
\end{equation*}
\end{definition}

In other words, the smooth family $(a_x)_{x\in M}$ of sesquilinear
forms $a_x\colon T_x^*M\otimes E \times T_x^*M\otimes E\to \CC$
satisfies the strong Legendre condition if, and only if, it is
\emph{uniformly strongly coercive}.

\begin{lemma}\label{lemma.eq.norms}
Let us assume that $(M, \pa_D M \sqcup \pa_R M)$ has finite width.  Then
the semi-norm
\begin{equation*} 
  ||| u |||^2 \define
  \|\nabla u\|_{L^2(M; E)}^2 + \int_{\pa_R M} |u|_E^2 \dvol_{\pa g}
\end{equation*}
is a norm on $H^{1}_{D}(M; E)$ that is equivalent to the $H^1$-norm.
\end{lemma}

\begin{proof}
Using the trace theorem \cite{GrosseSchneider2013}, there is $c > 0$
such that $|||u||| \le c \|u\|_{H^1}$.  The reverse inequality is
obtained as follows: Let $c_2$ be the best constant in the Poincar\'e
inequality of Corollary~\ref{cor.Poinc1} for $A = \pa_D M\cup \pa_R M$
and sections \emph{vanishing on $\pa_D M$.} Then $\|u\|_{H^{1}}^2 \leq
(1 + c_2) |||u|||^2$.
\end{proof}

The strong Legendre condition and Poincar\'e's inequality combine to
yield \emph{strong} coercivity:

\begin{proposition}\label{prop.coercive} 
Let $P = P_{(a, b)}$ be a second order (linear) differential operator
in divergence form on the vector bundle $E \to M$ (see Definition
\ref{def.divergence}). Assume that $(M, \pa_D M \sqcup \pa_R M)$ has
finite width, that $P$ satisfies the strong Legendre condition, that
$\Re\, b \define \frac12 (b + b^*) \ge 0$ on $\pa M$ (as operators),
and that there exists $\epsilon > 0$ such that $\Re\, b \geq \epsilon$
on $\pa_R M$, then $P$ is strongly coercive on $H^1_D(M; E)$. (So $Q =
Q_1 = 0$ in this result.)
\end{proposition}

\begin{proof}
The definition of $\tilde P_{(a, b)}$, Equation
\eqref{eq.def.Pab}, gives for all $u\in H^1_D(M;E)$ that
\begin{multline*}
  \Re\,  (\tilde P_{(a, b)} u)(u)= \int_M \Re\,  a(\nabla u, \nabla u)
  \dvol_g + \int_{\pa M \smallsetminus \pa_D M}  \Re\,  (bu, u) \dvol_{\pa g} \\
  \geq \, \gamma_a \|\nabla u\|^2 + \epsilon \int_{\pa_R M} |u|^2
  \dvol_{\pa g} \, \geq \, \min \{ \gamma_a, \epsilon \} |||u|||^2
  \, \geq \, \frac{ \min \{ \gamma_a, \epsilon \} }{1 + c_2} \,
  \gamma_a \| u \|_{H^1}^2,
\end{multline*}
where the last step is by Lemma~\ref{lemma.eq.norms}. The proof is
complete.
\end{proof}

The relation $\Re b \define \frac12(b+b^*) \geq \epsilon$, as operators,
means, as customary, that
\begin{equation*}
  \Re (b\zeta, \zeta) = (\Re b\zeta, \zeta) \geq \epsilon \Vert
  \zeta\Vert_{L^2}^2,
\end{equation*}
for all $\zeta\in H^1(\pa_R M ;E)$. 

We are interested in strongly coercive operators in view of the
Lax-Milgram Lemma (see, for example, \cite[Section
  5.8]{TrudingerBook}).

\begin{lemma}[Lax--Milgram lemma]\label{lemma.LaxMilgram}
Let $S \colon V \to V^{*}$ be a strongly coercive map with $\Re \<Su,
u\> \geq\gamma \|u\|_V^2$. Then $S$ is invertible and $\|S^{-1}\| \leq
\gamma^{-1}$.
\end{lemma}

Combining the above results (Proposition~\ref{prop.coercive} and the
Lax-Milgram Lemma \ref{lemma.LaxMilgram}), we immediately obtain the
following theorem which is the analog result of Theorem~\ref{thm.main}
for $k=0$.

The theorem uses the definitions of $P$ and $\tilde P$ explained in
Definition~\ref{def.divergence}.  Recall that $\tilde P_{(a, b)}$ is
defined by the sesquilinear form $a$, by the first order differential
operator $b$ acting on $E_{\pa_R M}$, by the first order differential
operators $Q$ and $Q_1$, and, finally, by the relation $\tilde P =
\tilde P_{(a, b)} + \tilde Q + \tilde Q_1^*$. All operators are
assumed to have bounded coefficients.  Moreover, $P_{(a, b)}$ is the
associated second order operator obtained by partial integration from
$\tilde P_{(a, b)}$ ignoring boundary terms, that is, $P = P_{(a, b)}
+ Q + Q_1^* $.

\begin{theorem}\label{thm.H1} 
Let $(M, g)$ be a Riemannian manifold with boundary. Assume that:
\begin{enumerate}[(i)]
 \item $(M, \pa_D M \sqcup \pa_R M)$ has finite width.

 \item $P = P_{(a,  b)} + Qu + Q_1^*$ satisfies the strong Legendre
   condition and has bounded coefficients, as usual;

 \item $\Re \, b \ge 0$ and there is $\epsilon >0$ such that $\Re b
   \geq \epsilon$ on $\pa_R M$.
   
 \item \label{Qcond} there is $\delta = \delta(a, b, g) \geq 0$ small
   enough such that $\Re (Q + Q_1) \geq - \delta$.

\end{enumerate}
Then $\tilde P \colon H^1_{D}(M; E) \to H^1_{D}(M; E)^{*}$ is an
isomorphism.
\end{theorem}

Note that $\Re(Q + Q_1^*) = \Re(Q + Q_1)$. In particular, the
condition $\Re (Q + Q_1) \geq - \delta$ means that
\begin{equation*}
  \Re ((Q + Q_1) \xi, \xi ) = \Re \big ( ( Q \xi, \xi ) + ( \xi
  , Q_1 \xi ) \big ) \geq - \delta \| \xi\|_{H^1}^2,
\end{equation*}
for all $\xi \in H^1_D(M; E)$.

\subsection{Higher regularity}\label{ssec.four}

We continue to assume that $M$ is a smooth manifold with smooth
boundary and bounded geometry. In this section, we record what is one
of our main applications of the Poincar\'e inequality, that is, the
well-posedness of the mixed Dirichlet-Robin problem on manifolds with
finite width in \emph{higher} Sobolev spaces. Even the particular case
of the Poisson problem with Neumann or Dirichlet boundary conditions
is new in the setting of manifolds with bounded geometry. These
results extend the well-posedness result in energy spaces of the
previous subsection to higher regularity Sobolev spaces. They follow
by combining the well-posedness in energy spaces with the regularity
results in \cite{GN17}.

To this end, we assume that $P$ has coefficients in $W^{k, \infty}$,
for some fixed $k \geq 1$.  We also continue to assume that $\tilde P
= \tilde P_{(a, b)} + \tilde Q + \tilde Q_1^*$ (again with $P$ and
$\tilde P$ defined as in Definition~\ref{def.divergence}) satisfies
the strong Legendre condition and $\Re\, b$ is strictly positive on
$\pa_R$ and nonnegative everywere. We have seen then that $\tilde P_{(a,
  b)}$ is strongly coercive.

Let us define
\begin{equation}\label{eq.def.jk}
 j_k \colon H^{k-1}(M; E) \oplus H^{k-1/2}(M \setminus \pa_D M; E) \to
 H^1_D(M; E)^*
\end{equation}
by $j_k(f, g)(w) \define \int_M (f, w) \dvol_g + 
  \int_{\partial M \smallsetminus \pa_D M}  (g, w) \dvol_{\pa g}$,
if $k \geq 1$, $j_0 = id$, if $k = 0$. Note, however, that, for $k =
0$, we have an exact sequence
\begin{equation*}
    0 \to H^{-1/2}(M \setminus \pa_D M; E) \to H^1_D(M; E)^* \to
    H^{-1}(M; E) \to 0,
\end{equation*}
which explains our notation. If $\tilde Pu = j_k(f, g)$, we shall
write $\pa_\nu^P u + bu = g$ and $Pu = f$.  This explains the
difference between $\tilde P$ and $P$. See \cite{GN17} for more
details.

The following result was proved in \cite[Corollary 7.5]{GN17}, using
that the Neumann and Robin problems satisfy regularity. See also
\cite{ADN1, ADN2, McLeanBook, Nirenberg55}.

\begin{theorem}\label{thm.regularity} \cite{GN17}
Assume that the operator $P = P_{(a, b)} + Q + Q_1$ satisfies the
strong Legendre condition, that it has coefficients in $W^{k,
  \infty}$, $k \geq 1$, and $\Re\, b$ is an order zero
  operator. Then there exists $c > 0$ such that
\begin{multline*}
   \| u \|_{H^{k+1}(M; E)} \, \leq c \ \big (\, \|P u\|_{H^{k-1}(M;
     E)} + \| u \|_{H^{1}(M; E)}\\ + \| u |_{\partialD M}
   \|_{H^{k+1/2}(\partialD M; E)} + \| \pa_\nu^P u + bu \|_{H^{k -
       1/2}(M \setminus \pa_D M; E)}\, \big) \ ,
\end{multline*}
for any $u \in H^{1}(M; E)$ such that $\tilde P u \in j_k(H^{k-1}(M;
E) \oplus H^{k-1/2}(\pa M\setminus \pa_D M; E))$. For $k = 0$ the
statement is trivial (once suitably reformulated). 
\end{theorem}

The meaning of this result is also that, if $u \in H^{1}(M; E)$,
$u|_{\partialD M} \in H^{k+1/2}(\pa_D M; E)$, and $\tilde Pu \in
\operatorname{Im}(j_k) = j_k(H^{k-1}(M; E) \oplus H^{k-1/2}(M
\setminus \pa_D M; E))$ with $\pa_\nu^P u + bu \in H^{k-1/2}(M
\setminus \pa_D M; E)$, then, in fact, $u \in H^{k+1}(M; E)$.

To prove Theorem \ref{thm.main}, we first notice that the
  assumption that $\Re\, b \ge 0$ implies that $\Re b  \define  \frac12 (b +
  b^*)$ is of order zero, since $b$ is of order (at most) one.
Theorem \ref{thm.main} is therefore a consequence of Theorems
\ref{thm.H1} and \ref{thm.regularity}.

\subsection{Applications and extensions}\label{ssec.ae}

We include now some consequences and extensions of our main result,
Theorem \ref{thm.main}. For simplicity, we assume here that our
  differential operators have totally bounded coefficients.

\subsubsection{Splitting of $E$} \label{sssec.splitting}
Let us assume that we are given a splitting
\begin{equation}\label{splitting.E}
  E\vert_{\pa M} = E_D \oplus E_R \oplus E_N
\end{equation}
as a direct sum of three smooth vector bundles with bounded geometry.
We denote by $p_D, p_R, p_N$ the associated orthogonal projections $E
\to E_D, E_R, E_N$.  We then replace the space $H^1_D(M; E)$ with
\begin{equation}\label{eq.dev.V}
  V \define \{ u \in H^1(M; E) \, \vert \ p_D u = 0 \}.
\end{equation}
Up until this point, we had $E_D \define E\vert_{\pa_D M}$, $E_R  \define 
E\vert_{\pa_R M}$, and $E_N \define E\vert_{\pa_N M}$. The more general
framework introduced here is needed in order to treat the
Hodge-Laplacian.

\subsubsection{Assumptions under the splitting of $E$}
\label{sssec.assumptions}
In general, here is how the assumptions change in the new setting:

\begin{enumerate}[(i)]
\item The Poincar\'e inequality becomes the \emph{assumption} that the
  modified norm
\begin{equation*}
  |||u|||^2 \define \|\nabla u\|^2_{M} + \|p_R u\|^2_{\pa M}
\end{equation*}
is equivalent to the $H^1$-norm on $V$.

\item We continue to assume that $P$ has coefficients in $W^{\ell,
  \infty}$.

\item The differential operator $b$ is then assumed to satisfy $\Re b
  \geq \epsilon p_R$ for some $\epsilon > 0$.
  
\item Also, we continue to assume that $\Re(Q + Q_1^*) \ge -\delta$,
  for some $\delta$ small enough, with $\delta$ depending on $a$,
  $\epsilon$, and $(M, g)$.
\end{enumerate}

Then Theorem \ref{thm.main} remains valid in this setting. This is
equivalent to Corollary \ref{cor.well.posedness.bvp_2} formulated in
detail below. Before discussing this theorem, let us notice that
condition (i) replacing the Poincar\'e inequality is somewhat tricky,
as seen in the following example.

\begin{example}
Let $M = [0, 1]$ with the standard, euclidean metric. Then both
  $(M, \{0\})$ and $(M, \{1\})$ are of finite width, so they satisfy
  the Poincar\'e inequality (for scalar functions, that is, for $E =
  \CC$). Let now $E = \CC^2$. It is enough to take $E_R = \{0\}$, but $V$
  as in \eqref{eq.dev.V}. We thus need to specify $E_D$ above $\pa M = \{0,
  1\}$. Two seemingly similar choices will give completely different
  results.

  Indeed, let $E_D = \{0\} \oplus \CC$ above $\{0\}$. Then Assumption (i)
  on the equivalence of norms is satisfied if $E_D = \CC \oplus \{0\}$
  above $\{1\}$, but is not satisfied if $E_D = \{0\} \oplus \CC$ above
  $\{1\}$. The first case corresponds to putting together $(M, \{0\})$
  and $(M, \{1\})$, whereas the second case corresponds to putting
  together $(M, \{0, 1\})$ and $(M, \emptyset)$. In the second case,
  the Poincar\'e inequality is clearly not satisfied (since $u = 1$ is
  allowed).
\end{example}

\subsubsection{Boundary value problems}
Recall the discussion on boundary value problems in Section~\ref{sssec.bvp}.
As usual, Theorem \ref{thm.main} gives results on boundary value
problems.  We formulate, nevertheless, the result in the more general
framework relying on a decomposition of $E$ as in Equation~\eqref{splitting.E}.

\begin{corollary}\label{cor.well.posedness.bvp_2}
We consider the setting of Section~\ref{sssec.assumptions}. Then the boundary
  value problem
\begin{align}\label{eq.bvp}
\left\{ \begin{array}{rlll} P u &= \ f\, \in H^{\ell-1}(M; E) &&
  \text{ in } M\\
   p_D u &= \ h_0 \in H^{\ell+1/2}(\pa M; E_D) &&\text{ on } \pa M \\
% %
   (1- p_D) (\partial_\nu^P u + bu) &= \ h_1 \in H^{\ell-1/2}(\pa M;
   E_R \oplus E_N) && \text{ on } \pa M
\end{array}\right.
\end{align}
is well-posed (\ie  it has a unique solution $u \in H^{\ell+1}(M; E)$
that depends continuously on $h_0$ and $h_1$).
\end{corollary}

\subsubsection{Self-adjointness} As in \cite{AGN1}, we obtain the
following corollary.

\begin{corollary}\label{cor.sa} Let us assume that $P$ is as in 
Section~\ref{sssec.assumptions} and, moreover, that it has coefficients in
$W^{1, \infty}$ and is formally self-adjoint, that is, that $(Pu, v)
  = (u, Pv)$ for $u, v \in \CIc(M\smallsetminus \pa M; E)$.  Then $P$
  with domain $$\maD(P) \define  \{ u \in H^2(M; E) \, \vert \ p_D u = 0,
  \ (1 - p_D) (\partial_\nu^P u + bu) = 0\}$$ is self-adjoint.
\end{corollary}

See also \cite{DancerDaners, Daners2000, gesztesyMitrea2009,
  gesztesyMitrea2014a}, where bounded domains, but with Lipschitz or
more general boundaries, were considered. As in those papers, one
obtains also consequences for the corresponding parabolic equations.

\subsubsection{Coercivity in general and G\aa rding's inequality}
\label{sssec.coercive}
As is well known, results such as Corollary \ref{cor.sa} are closely
related to G\aa rding's inequality. This inequality is usually
obtained for uniformly strongly elliptic operators. Indeed, following
\cite{AgmonCoercive}, we can extend our results to uniformly strongly
elliptic operators as follows.

Recall that an operator $P$ is \emph{coercive} on $V \subset H^1(M;
E)$ if it satisfies the G{\aa}rding inequality, that is, if there
exist $\gamma > 0$ and $R \in \RR$ such that for all $u\in V$
\begin{equation}\label{eq.Gaarding}
  \Re (Pu, u) \geq \gamma \|u\|_{H^1(M;E)}^2 - R \|u\|^2_{L^2(M; E)}.
\end{equation}
Then $P + \lambda$ is strongly coercive for $\Re(\lambda) > R$, and
hence Theorems \ref{thm.main} and \ref{thm.H1} remain true for $P$
replaced with $P + \lambda$. Coercive operators on \emph{bounded}
domains were characterized by Agmon in \cite{AgmonCoercive} as
strongly elliptic operators satisfying suitable conditions at the
boundary (which we shall call the ``Agmon condition.''). We shall need
a \emph{uniform} version of this condition, to account for the
non-compactness of the boundary.

Let $P^{(0)}_x$ be the principal part of the operator $P$ and
$C_x^{(0)}$ be the principal part of the boundary conditions
$(p_D, (1-p_D) (\partial_\nu^P + b))$ with coefficients frozen at some
$x \in \pa M$, as in \cite{GN17}. Let $B_x^{(0)}$ be the associated
Dirichlet bilinear form to $P^{(0)}_x$ equipped with the above
  boundary conditions (again with coefficients frozen at $x$). This
is as in Equation \eqref{eq.def.B}. In particular, we have the
projection $p_{D,x}^{(0)} \colon E_x \to (E_D)_x$ that enters in the
boundary conditions defined by $C_x^{(0)}$. This defines a boundary
value problem on the half-space $T_x^+M$ and a bilinear form on
$T_x^+M$ that is continuous in the $H^1$-seminorm $\vert
  u\vert_{H^1}\define \Vert \nabla u\Vert_{L^2}$.

\begin{definition}
We say that $P$ (or the form $B$ of Equation \eqref{eq.B.tildeP})
satisfies the \emph{uniform Agmon condition (on $\pa M$)} if it is
uniformly strongly elliptic and if there exists $C > 0$ such
\begin{equation*}
   B_x^{(0)} (u, u)  \define  (P^{(0)}_x u, u) + \int_{T_x \pa M}
    (b^{(0)} u, u) dx' \geq C |u|_{H^1}^2,
\end{equation*}
for all $x \in \pa M$ and all $u \in \CIc(T_x^+M)$ that satisfies
$p_{D,x}^{(0)} u = 0$ (on $T_x \pa M = \pa T_x^+M$).
\end{definition}

We have then the following result that is proved, \emph{mutatis
  mutandis}, as the regularity result in \cite{GN17}, to which we
refer for more details.

\begin{theorem}
We use the notation in \ref{sssec.splitting}, in particular,
$$V \define \{ u \in H^1(M; E) \, | \ p_D u = 0 \}.$$ We have that $\tilde
P$ (equivalently, the form $B$ of Equation \eqref{eq.B.tildeP}) is
coercive on $V$ if, and only if, it satisfies the uniform Agmon
condition on $\pa M$.
\end{theorem}

The idea of the proof, in one direction, is to consider $u$ with a
shrinking supports towards $x$ using dilations and to retain the
dominant terms. In the other direction, one uses the standard
partitions of unity on manifolds with (boundary and) bounded geometry.
See \cite{GN17, Taylor1} for details of this method.

\begin{remark}
The reader may have noticed that our Robin boundary conditions are of
the form $\pa_\nu^P + b$. It makes sense, of course, to consider
boundary conditions of the form $a\pa_\nu^P + b$, where~$a$ is an
endomorphism of $E_R \oplus E_N$. If $a$ is invertible, this changes
nothing. However, significant differences arise if $a$ is
singular. See, for instance, the recent preprint \cite{nazarofPopoff}
and the references therein.
\end{remark}

See \cite{BaerBallmann, KarstenBMB, PlamenevskiBook, SchSch1,
  SchroheSeiler} for an approach to boundary value problems on
non-compact manifolds using pseudodifferential operators and for
related recent results.

%%%%
\def\cprime{$'$}

\end{document}